\documentclass{amsart}
\usepackage{latexsym, amssymb, amsmath, amsthm}
\usepackage[T1]{fontenc}
\usepackage{lmodern}
\usepackage{enumerate}
\usepackage{mathabx}
\usepackage{csquotes}
\usepackage[mathscr]{euscript}
\usepackage{parskip}
\usepackage{thmtools, thm-restate}
\usepackage{bussproofs}
\usepackage{xcolor}

\usepackage{bm}

\makeatletter
\newlength\knuthian@fdfive
\def\mathpal@save#1{\let\was@math@style=#1\relax}
\def\utilde#1{\mathpalette\mathpal@save
              {\setbox124=\hbox{$\was@math@style#1$}%
\setbox125=\hbox{$\fam=3\global\knuthian@fdfive=\fontdimen5\font$}
\setbox125=\hbox{$\widetilde{\vrule height 0pt depth 0pt width \wd124}$}%
               \baselineskip=1pt\relax
               \lineskiplimit=\z@\relax
               \lineskip=1pt\relax
               \vtop{\copy124\copy125\vskip -\knuthian@fdfive}}}
\makeatother

\usepackage{pifont}

\declaretheorem[numberwithin=section]{theorem}
\newtheorem{lemma}[theorem]{Lemma}
\newtheorem{corollary}[theorem]{Corollary}

\newtheorem{question}[theorem]{Question}
\newtheorem*{claim}{Claim}
\newtheorem*{theorem*}{Theorem}

\theoremstyle{definition}
\newtheorem{definition}[theorem]{Definition}
\theoremstyle{remark}
\newtheorem{remark}[theorem]{Remark}

\title{An incompleteness theorem via ordinal analysis}
\author{James Walsh}
\thanks{Thanks to Dan Appel, Reid Dale, and Paolo Mancosu for discussion. Thanks also to an anonymous referee for helpful comments and suggestions. Special thanks to Anton Freund for discussion of Section \ref{avoiding-diag} and to Antonio Montalb\'{a}n for discussion of Section \ref{main-section}.}
\address{Sage School of Philosophy, Cornell University}
\email{jameswalsh@cornell.edu}

\begin{document}

\maketitle

\begin{abstract}
We present an analogue of G\"{o}del's second incompleteness theorem for systems of second-order arithmetic. Whereas G\"{o}del showed that sufficiently strong theories that are $\Pi^0_1$-sound and $\Sigma^0_1$-definable do not prove their own $\Pi^0_1$-soundness, we prove that sufficiently strong theories that are $\Pi^1_1$-sound and $\Sigma^1_1$-definable do not prove their own $\Pi^1_1$-soundness. Our proof does not involve the construction of a self-referential sentence but rather relies on ordinal analysis.
\end{abstract}

\section{Introduction}

The motivation for this project come from two sources: G\"{o}del's second incompleteness theorem and Gentzen's consistency proof of arithmetic. These results are complementary in many ways. In the first place, they jointly form a complicated and ambiguous resolution of Hilbert's problem of proving the consistency of arithmetic. Moreover, Gentzen's proof refines G\"{o}del's result by exhibiting the first example of a non-meta-mathematical arithmetic statement---namely, the statement that $\varepsilon_0$ lacks primitive recursive descending sequences---that is not provable from the Peano axioms. Though G\"{o}del's result is highly general, his proof relies on self-reference, rendering it opaque and mysterious \cite{kotlarski2004incompleteness, salehi2020diagonal, visser2019tarski}. By contrast, Gentzen's proof is concrete but his results are specific to the case of Peano arithmetic. In this paper we prove a version of the second incompleteness theorem that is general like G\"{o}del's but with a proof that is concrete like Gentzen's; in particular, we use the methods of ordinal analysis and do not rely on diagonalization or self-reference.

Let's start by giving a typical statement G\"{o}del's second incompleteness theorem:

\begin{theorem}[G\"{o}del]
No consistent and recursively axiomatizable extension of elementary arithmetic proves its own consistency.
\end{theorem}

Recursive axiomatizability is equivalent to $\Sigma^0_1$-definability by Craig's Trick. Moreover, consistency is provably equivalent (in elementary arithmetic) to $\Pi^0_1$-soundness. Hence, we may restate G\"{o}del's Theorem as follows:

\begin{theorem}[G\"{o}del]\label{godel}
If $T$ is a $\Pi^0_1$-sound and $\Sigma^0_1$-definable extension of elementary arithmetic, then $T$ does not prove its own $\Pi^0_1$-soundness.
\end{theorem}

We prove the following analogous result for systems of second-order arithmetic:
\begin{theorem}\label{main}
If $T$ is a $\Pi^1_1$-sound and $\Sigma^1_1$-definable extension of $\Sigma^1_1\text{-}\mathsf{AC}_0$, then $T$ does not prove its own $\Pi^1_1$-soundness.
\end{theorem}
$\Pi^1_1$-soundness is a strictly stronger condition than $\Pi^0_1$-soundness. However, $\Sigma^1_1$-definability is a strictly weaker condition than $\Sigma^0_1$-definability. Hence, Theorem \ref{main} is neither weaker nor stronger than G\"{o}del's Theorem but incomparable with it.


Let's take a brief look at the ideas motivating the proof. In what follows $\mathsf{WF}(\prec)$ is a sentence expressing the well-foundedness of $\prec$:
$$\mathsf{WF}(\prec):= \forall X\big(\exists x\in X \to \exists x\in X \; \forall y \in X \; \neg y \prec x \big) ;$$
$\mathsf{RFN}_{\Pi^1_1}(T)$ is a sentence naturally expressing the $\Pi^1_1$-soundness of $T$:
 $$\mathsf{RFN}_{\Pi^1_1}(T) := \forall \varphi\in\Pi^1_1 \big( \mathsf{Pr}_T(\varphi) \to \mathsf{True}_{\Pi^1_1}(\varphi) \big);$$ and the proof-theoretic ordinal $|T|_{\mathsf{AN}}$ of a theory $T$ is the supremum of the ordinals $\alpha$ for which there is some $\Sigma^1_1$ presentation $\prec$ of $\alpha$ such that $T\vdash \mathsf{WF}(\prec)$. 
 
 Assuming that $T$ is $\Pi^1_1$-sound and $\Sigma^1_1$-definable, Spector's $\Sigma^1_1$-bounding theorem implies that $|T|_{\mathsf{AN}}$ is strictly less than $\omega_1^{\mathsf{CK}}$, whence $|T|_{\mathsf{AN}}$ has some $\Sigma^1_1$ presentation. For any $\Sigma^1_1$-presentation $\prec$ of $|T|_{\mathsf{AN}}$, the following is true by definition:
\begin{equation}
  T \nvdash \mathsf{WF}(\prec).
\end{equation}
We then need to show that there is \emph{at least one} $\Sigma^1_1$-presentation $\prec$ of $|T|_{\mathsf{AN}}$ such that:
\begin{equation}
 T\vdash \mathsf{RFN}_{\Pi^1_1}(T)\to\mathsf{WF}(\prec).
\end{equation}
For then, from claims (1) and (2), we infer that $T\nvdash \mathsf{RFN}_{\Pi^1_1}(T)$.


This proof is analogous to a folklore proof of a different version of the second incompleteness theorem, namely, that no $\Pi^0_2$-sound and $\Sigma^0_1$-definable theory $T$ proves its own $\Pi^0_2$-soundness. In this folklore proof one first defines a recursive function $f_T$ that is not provably total in $T$ by diagonalizing against the set of provably total recursive functions of $T$; one then shows that the totality of $f_T$ is $T$-provably equivalent to the $\Pi^0_2$-soundness of $T$. See \cite{beklemishev1997induction} for a detailed proof. Just as the class of the provably total recursive functions of $T$ is the canonical ``invariant'' measuring the $\Pi^0_2$-strength of $T$, the proof-theoretic ordinal of $T$ is the canonical ``invariant'' measuring the $\Pi^1_1$-strength of $T$. And just as the non-provability of $\Pi^0_2$-soundness is derived by defining a total recursive function in terms of the canonical $\Pi^0_2$-invariant, we derive Theorem \ref{main} by defining a well-ordering in terms of the canonical $\Pi^1_1$-invariant.

A slight modification of our proof of Theorem \ref{main} delivers a stronger result:
\begin{theorem}\label{stronger}
There is no sequence $(T_n)_{n<\omega}$ of $\Pi^1_1$-sound and $\Sigma^1_1$-definable extensions of $\Sigma^1_1\text{-}\mathsf{AC}_0$ such that for each $n$, $T_n \vdash \mathsf{RFN}_{\Pi^1_1}(T_{n+1})$.
\end{theorem}

To see that Theorem \ref{stronger} implies Theorem \ref{main}, note that if $T$ were a counter-example to Theorem \ref{main} then we would get a counter-example to Theorem \ref{stronger} by letting $T=T_n$ for each $n$. Theorem \ref{stronger} extends earlier work \cite{pakhomov2018reflection, pakhomov2021reflection} of Pakhomov and the author, who proved the following:
\begin{theorem}[Pakhomov--W.]\label{pakhomov-first}
There is no sequence $(T_n)_{n<\omega}$ of $\Pi^1_1$-sound and $\Sigma^0_1$-definable extensions of $\mathsf{ACA}_0$ such that for each $n$, $T_n \vdash \mathsf{RFN}_{\Pi^1_1}(T_{n+1})$.
\end{theorem}
Pakhomov and the author proved Theorem \ref{pakhomov-first} to provide an explanation for the apparent pre-well-ordering of natural theories by proof-theoretic strength; see \cite{walsh2021hierarchy} for a discussion of this phenomenon. Theorem \ref{stronger} extends this explanation to the new setting of $\Sigma^1_1$-definable theories. Whereas the proof of Theorem \ref{pakhomov-first} appeals to the second incompleteness theorem and makes no mention of proof-theoretic ordinals, our proof of Theorem \ref{stronger} uses ordinal analysis and does not appeal to any version of the second incompleteness theorem. Note that Theorem \ref{stronger} is neither stronger nor weaker than Theorem \ref{pakhomov-first}; the former requires the stronger hypothesis that $T$ extend $\Sigma^1_1\text{-}\mathsf{AC}_0$ but only the weaker hypothesis that $T$ be $\Sigma^1_1$-definable.

The main tool that we use to derive Theorem \ref{main} and Theorem \ref{stronger} is Spector's $\Sigma^1_1$-bounding theorem. Though the standard proofs (e.g., \cite{sacks2017higher} Chapter 1, Corollary 5.5) of Spector's theorem rely on diagonalization, there is an alternate diagonalization-free proof due to Beckmann and Pohlers \cite{beckmann1998applications}. This latter proof derives $\Sigma^1_1$-bounding from an analysis of cut-free infinitary derivations. In particular, $\Sigma^1_1$-bounding is derived from a result known as ``the boundedness theorem,'' which roughly states that for arithmetically definable well-orders $\prec$, the order-type of $\prec$ cannot exceed the depth of the shortest proof of the well-foundedness of $\prec$ in $\omega$-logic. Versions of the boundedness theorem are already implicit in Gentzen's proof \cite{gentzen1969provability} that $\mathsf{PA}$ does not prove the primitive recursive well-foundedness of $\varepsilon_0$. Note that Gentzen's proof of this independence result does \emph{not} appeal to G\"{o}del's second incompleteness theorem and does not rely on self-reference but rather involves a combinatorial analysis of proofs in $\mathsf{PA}$.

We also rely on a \emph{formalized} version of Spector's theorem. Roughly, the formalized version says that any for any $\Sigma^1_1$ predicate $H$, if $\mathsf{ACA}_0$ proves ``$H$ is a set of recursive ordinals,'' then for some $e$, $\mathsf{ACA}_0$ proves ``$e$ is a recursive ordinal but $\neg H(e)$.'' The standard proof of the formalized version of $\Sigma^1_1$-bounding uses the recursion theorem to define a recursive function. Though this is not exactly the construction of a self-referential sentence, definitions using the recursion theorem share the opacity of constructions of sentences using the fixed point lemma. Accordingly, we provide a new proof of this formalized version of Spector's theorem. The new proof uses the same techniques Gentzen used to prove the boundedness theorem and that Beckmann and Pohlers used to prove the $\Sigma^1_1$-bounding theorem. Thus, we do not rely on self-reference or diagonalization in any form.


Here is our plan for the rest of the paper. In \textsection \ref{prelims} we cover some preliminary material, including definitions and notation; we also discuss two folklore results that we use to prove the main theorems. In \textsection \ref{main-section} we provide proofs of Theorem \ref{main} and Theorem \ref{stronger}. In \textsection \ref{avoiding-diag} we provide an alternate proof using infinitary derivations of the formalized version of Spector's $\Sigma^1_1$-bounding theorem. Finally, in \textsection \ref{open-problems} we present some open problems concerning the optimality of Theorem \ref{main}.


\section{Preliminaries}\label{prelims}

One of the central concepts in proof theory is that of a proof-theoretic ordinal. To say what proof-theoretic ordinals are, we must say what a presentation of an ordinal is.

\begin{definition}
For a syntactic complexity class $\Gamma$, a $\Gamma$ \emph{presentation} of an ordinal $\alpha$ is a $\Gamma$ formula that defines an ordering of order-type $\alpha$ over the standard structure $\big(\mathbb{N},\mathcal{P}(\mathbb{N})\big)$.
\end{definition}

We now present two definitions of ``proof-theoretic ordinal,'' both of which are necessary for our proof.

\begin{definition}
Let $|T|_{\mathsf{RE}}$ be the supremum of the ordinals $\alpha$ for which there is some $\Sigma^0_1$ presentation $\prec$ of $\alpha$ such that $T\vdash \mathsf{WF}(\prec)$.
\end{definition}

\begin{definition}
Let $|T|_{\mathsf{AN}}$ be the supremum of the ordinals $\alpha$ for which there is some $\Sigma^1_1$ presentation $\prec$ of $\alpha$ such that $T\vdash \mathsf{WF}(\prec)$.
\end{definition}

The $\mathsf{AN}$ in the notation $|T|_{\mathsf{AN}}$ means \emph{analytic}. Indeed, $|T|_{\mathsf{AN}}$ is the supremum of the $T$-provably well-founded \emph{analytic} linear orders, where analytic means lightface $\Sigma^1_1$.

By definition, $|T|_{\mathsf{RE}}\leq |T|_{\mathsf{AN}} \leq  \omega_1^{\mathsf{CK}}$. However, we can say more about the relationship between these three values if we make some assumptions about $T$. In the following subsections we will describe these results, which belong to mathematical folklore.

\subsection{The first folklore result}\label{first-folk}

If $T$ is $\Pi^1_1$-sound and $\Sigma^1_1$-definable, we can say more about the relationship between $|T|_{\mathsf{AN}}$ and $\omega_1^{\mathsf{CK}}$.

\begin{theorem}[Folklore]\label{bounding}
If $T$ is $\Pi^1_1$-sound and $\Sigma^1_1$-definable, then $|T|_{\mathsf{RE}}<\omega_1^{\mathsf{CK}}$.
\end{theorem}

Theorem \ref{bounding} follows immediately from Spector's $\Sigma^1_1$-bounding  theorem:

\begin{theorem}[Spector]
For any $\Sigma^1_1$ presentation $\prec$ of an ordinal, $\mathsf{otyp}(\prec)<\omega_1^{\mathsf{CK}}$.
\end{theorem}

Standard proofs of Spector's $\Sigma^1_1$-bounding appeal to the fact that Kleene's $\mathcal{O}$ is not $\Sigma^1_1$-definable. Note that the latter is typically proved using an ordinary diagonalization argument; see, e.g., \cite{sacks2017higher} Chapter 1, Theorem 5.4.

However, there is an alternate proof due to Beckmann and Pohlers \cite{beckmann1998applications} of Spector's Theorem that does \emph{not} use diagonalization. The Beckmann--Pohlers proof proceeds by analyzing the structure of infinitary cut-free derivations. Beckmann and Pohlers derive $\Sigma^1_1$-bounding a result known as ``the boundedness lemma,'' which they claim is essentially implicit in Gentzen's proof of the $\mathsf{PA}$ non-derivability of $\varepsilon_0$-induction. Accordingly, when we appeal to Theorem \ref{bounding}, we are appeal to a result that has a Gentzen-style non-diagonalization proof.

\subsection{The second folklore result}

We can say more about the relationship between $|T|_{\mathsf{RE}}$ and $|T|_{\mathsf{AN}}$ for all $\Pi^1_1$-sound $T$ that extend $\Sigma^1_1\text{-}\mathsf{AC}_0$. Recall that $\Sigma^1_1\text{-}\mathsf{AC}_0$ is the theory whose axioms are those of $\mathsf{ACA}_0$ plus each instance of the schema:
$$\forall n \exists X \varphi(n,X) \to \exists Y\forall n \varphi\big(n, (Y)_n\big)$$
where $\varphi(n,X)$ is a $\Sigma^1_1$ formula in which $Y$ does not occur and where
$$(Y)_n = \{ i \mid (i,n)\in Y\}.$$

\begin{theorem}[Folklore]\label{folklore}
If $T$ is a $\Pi^1_1$-sound extension of $\Sigma^1_1\text{-}\mathsf{AC}_0$, then $|T|_{\mathsf{RE}}=|T|_{\mathsf{AN}}$.
\end{theorem}

The main tool for proving Theorem \ref{folklore} is a \emph{formalized} version of Spector's theorem. To state this formalized result, we first introduce some notation.

\begin{definition}
Let $Rec:=\{e\in\mathbb{N} \mid e \text{ is an index of a total recursive function}\}$. With each $e\in Rec$ there is an associated a relation $\prec_e$ where $$n\prec_e m :\Leftrightarrow \{e\}(\langle n,m\rangle)=0$$ where $\langle, \rangle$ is a primitive recursive pairing function.
\end{definition}

\begin{definition}
Let $\mathfrak{W}_{Rec} : = \{e\in\mathbb{N} \mid e\in Rec \wedge \mathsf{WO}(\prec_e)\}$.
\end{definition}

In \cite{rathjen1999realm} (see Proposition 2.19), Rathjen derives Theorem \ref{folklore} from the following lemma, which is a formalized version of Spector's $\Sigma^1_1$-bounding theorem:

\begin{lemma}[Rathjen]\label{rathjen} Suppose $H(x)$ is a $\Sigma^1_1$ formula such that:
$$\mathsf{ACA}_0 \vdash \forall x \big(H(x) \to x\in \mathfrak{W}_{Rec} \big).$$
Then for some $e\in Rec$:
$$\mathsf{ACA}_0 \vdash e \in \mathfrak{W}_{Rec} \wedge \neg H(e).$$
\end{lemma}

Theorem \ref{folklore} is straightforwardly derived from Lemma \ref{rathjen}. Rathjen provides a proof of Lemma 1.1 in \cite{rathjen1991role}. Note that this proof of Lemma \ref{rathjen} makes use of the recursion theorem. In Section \ref{avoiding-diag} we will present an alternative proof of Lemma \ref{rathjen} that does not make any use of the recursion theorem or other diagonalization.


\subsection{Remarks}

Before continuing, let's highlight some features of these folklore results and their relationship to the main theorems.

First, note the role that $\Sigma^1_1$-bounding plays in the proofs of the folklore results. We will not mention $\Sigma^1_1$-bounding explicitly in the proofs of the main theorems, but we will still rely on it insofar as it is used to prove these folklore results. We feel that the role of $\Sigma^1_1$-bounding is so important that it is worth explicitly highlighting where it is being used. 

Second, note that in the proofs of both folklore results, we must appeal to the $\Pi^1_1$-soundness of $T$. We will not appeal to $\Pi^1_1$-soundness explicitly in the proofs of the main theorems; we will only rely on it insofar as we invoke these folklore results.


One can find proofs of Theorem \ref{bounding} and Theorem \ref{folklore} in \cite{rathjen1999realm}.


\section{The Main Theorems}\label{main-section}

In this section we prove our main theorem, an analogue of G\"{o}del's second incompleteness theorem. We start by introducing two formulas and make a remark about their syntactic complexity. We will use these formulas and appeal to the remark many times, so it is worth isolating them here.

\begin{definition}\label{linear}
For a binary formula $\triangleleft$, let $\mathsf{LO}(\triangleleft)$ stand for the conjunction of the following clauses:
\begin{enumerate}
    \item $\neg \exists x  \mathsf{True}_{\Sigma^0_1}(x\triangleleft x)$
    \item $\forall x \forall y \big( \mathsf{True}_{\Sigma^0_1}(x\triangleleft y) \vee \mathsf{True}_{\Sigma^0_1}(y\triangleleft x) \vee x=y \big)$
    \item $\forall x \forall y \forall z \Big( \big( \mathsf{True}_{\Sigma^0_1}(x\triangleleft y) \wedge \mathsf{True}_{\Sigma^0_1}(y\triangleleft z) \big) \to \mathsf{True}_{\Sigma^0_1}(x\triangleleft z) \Big)$
\end{enumerate}
\end{definition}

\begin{definition}\label{well-founded}
For a binary formula $\triangleleft$, let $\mathsf{WF}(\triangleleft)$ stand for:
$$\forall X\big(\exists x\in X \to \exists x\in X \; \forall y \in X \; \neg \mathsf{True}_{\Sigma^0_1}(y \triangleleft x) \big)$$
\end{definition}

\begin{remark}
Note the use of the $\Sigma^0_1$ truth-predicate in Definition \ref{linear} and Definition \ref{well-founded}. Thus, for any formula $\triangleleft$, $\mathsf{LO}(\triangleleft)$ is arithmetic and $\mathsf{WF}(\triangleleft)$ is $\Pi^1_1$. Of course, $\mathsf{LO}(\triangleleft)$ and $\mathsf{WF}(\triangleleft)$ will make the most sense when applied to $\Sigma^0_1$ formulas or in quantified statements about $\Sigma^0_1$ formulas. We shall use it in the latter way.
\end{remark}

\subsection{The key lemma}

The key to the proofs of Theorem \ref{main} and Theorem \ref{stronger} is the following lemma:

\begin{lemma}\label{key-lemma}
If $T$ is $\Pi^1_1$-sound and $\Sigma^1_1$-definable, then there is a $\Sigma^1_1$ presentation $\prec_T$ of $|T|_{\mathsf{RE}}$ such that $\Sigma^1_1\text{-}\mathsf{AC}_0\vdash \mathsf{RFN}_{\Pi^1_1}(T) \to \mathsf{WF}(\prec_T).$
\end{lemma}

\begin{proof}
Let $T$ be $\Pi^1_1$-sound and $\Sigma^1_1$-definable. By Theorem \ref{bounding}, $|T|_{\mathsf{RE}}<\omega_1^{\mathsf{CK}}$, whence there is some $\Sigma^0_1$-definable $\prec_\star$ such that $|T|_{\mathsf{RE}}=\mathsf{otyp}(\prec_\star)$.

We are now going to define an alternate presentation $\prec_T$ of $|T|_{\mathsf{RE}}$. Informally, the formula $\alpha\prec_T \beta$ says that $\alpha$ is less than $\beta$ in an initial segment of the $\prec_\star$ ordering that embeds into a $\Sigma^0_1$-definable linear order $\triangleleft$ such that $T$ proves the well-foundedness of $\triangleleft$. More formally, we define $\alpha \prec_T \beta $ as the conjunction of:
\begin{enumerate}
    \item $\alpha\prec_\star\beta$
    \item $\exists \triangleleft \in \Sigma^0_1 \; \exists f  \; \Big( \mathsf{Emb}(f,\prec_\star\restriction \beta, \triangleleft) \; \wedge \mathsf{LO}(\triangleleft) \; \wedge \; \mathsf{Pr}_T\big(\mathsf{WF}(\triangleleft)\big) \Big).$
\end{enumerate}
where $\mathsf{Emb}(f,\prec_\star\restriction \beta,\triangleleft)$ stands for:
$$\forall x \forall y \Big( ( y\preceq_\star \beta \wedge x\prec_\star y) \to \mathsf{True}_{\Sigma^0_1}\big(f(x)\triangleleft f(y) \big) \Big)$$
and where $\mathsf{LO}(\triangleleft)$ and $\mathsf{WF}(\triangleleft)$ are as in Definition \ref{linear} and Definition \ref{well-founded}.

\begin{claim}
$\prec_T$ is $\Sigma^1_1\text{-}\mathsf{AC}_0$-provably equivalent to a $\Sigma^1_1$ formula.
\end{claim}

Clearly $\alpha \prec_\star \beta$ is $\Sigma^0_1$. Now let's look at the second conjunct of $\alpha\prec_T\beta$. Note that $\mathsf{Emb}(f,\prec_\star\restriction \beta,\triangleleft)$ and $\mathsf{LO}(\triangleleft)$ are both arithmetic. On the other hand, $\mathsf{Pr}_T\big( \mathsf{WF}(\triangleleft)\big)$ is $\Sigma^1_1$, since $T$ is $\Sigma^1_1$-definable. So the conjunction:
$$ \mathsf{Emb}(f,\prec_\star\restriction \beta, \triangleleft) \; \wedge \mathsf{LO}(\triangleleft) \; \wedge \; \mathsf{Pr}_T\big(\mathsf{WF}(\triangleleft)\big) $$
is provably equivalent in $\Sigma^1_1\text{-}\mathsf{AC}_0$ to a $\Sigma^1_1$ formula. Thus, the second conjunct of  $\alpha\prec_T\beta$ is given by an existential number quantifier before an existential set quantifier before a (formula that is $\Sigma^1_1\text{-}\mathsf{AC}_0$-provably equivalent to a) $\Sigma^1_1$ formula. It follows that $\prec_T$ is $\Sigma^1_1\text{-}\mathsf{AC}_0$-provably equivalent to a $\Sigma^1_1$ formula.

\begin{claim}
$\prec_T$ is a presentation of $|T|_{\mathsf{RE}}$.
\end{claim}

$\mathsf{otyp}(\prec_T) \leq |T|_{\mathsf{RE}}$: The first conjunct in the definition of $\prec_T$ ensures that $\mathsf{otyp}(\prec_T)\leq \mathsf{otyp}(\prec_\star)$. To finish the argument, recall that $\mathsf{otyp}(\prec_\star)=|T|_{\mathsf{RE}}$.

$\mathsf{otyp}(\prec_T) \geq |T|_{\mathsf{RE}}$: Let $\alpha<|T|_{\mathsf{RE}}=\mathsf{otyp}(\prec_\star)$. We need to see that $\alpha<\mathsf{otyp}(\prec_T)$.

Since $\alpha<\mathsf{otyp}(\prec_\star)$ and $\alpha<|T|_{\mathsf{RE}}$, there is an embedding of an initial segment of $\prec_\star$ that includes the $\prec_\star$ representation of $\alpha$ into a $\Sigma^0_1$-definable well-order that is $T$-provably well-founded. It is then immediate from the definition of $\prec_T$ that $\alpha<\mathsf{otyp}(\prec_T)$.

\begin{claim}
$\Sigma^1_1\text{-}\mathsf{AC}_0\vdash \mathsf{RFN}_{\Pi^1_1}(T) \to \mathsf{WF}(\prec_T).$
\end{claim}

Reason in $\Sigma^1_1\text{-}\mathsf{AC}_0$: Suppose that $\prec_T$ is ill-founded. Then, by the definition of $\prec_T$, there is some infinite descending sequence in $\prec_\star$ that embeds into a $\Sigma^0_1$-definable linear order $\triangleleft$ such that $T\vdash \mathsf{WF}(\triangleleft)$. Since $\triangleleft$ embeds an ill-founded linear order, $\triangleleft$ is ill-founded. So $T$ proves a false $\Pi^1_1$ sentence, namely, $\mathsf{WF}(\triangleleft)$.
\end{proof}

\subsection{An incompleteness theorem}

We now present a proof of Theorem \ref{main}, restated here:

\begin{theorem}\label{main-new}
If $T$ is a $\Pi^1_1$-sound and $\Sigma^1_1$-definable extension of $\Sigma^1_1\text{-}\mathsf{AC}_0$, then $T$ does not prove its own $\Pi^1_1$-soundness.
\end{theorem}

\begin{proof}
By Lemma \ref{key-lemma}, there is a $\Sigma^1_1$ presentation $\prec_T$ of $|T|_{\mathsf{RE}}$ such that:
$$\Sigma^1_1\text{-}\mathsf{AC}_0\vdash \mathsf{RFN}_{\Pi^1_1}(T) \to \mathsf{WF}(\prec_T).$$

Since $T$ extends $\Sigma^1_1\text{-}\mathsf{AC}_0$, we infer that:
\begin{equation}\label{implies-formula}
    T\vdash \mathsf{RFN}_{\Pi^1_1}(T) \to \mathsf{WF}(\prec_T).
\end{equation}

\begin{claim}
$T\nvdash \mathsf{WF}(\prec_T)$.
\end{claim}

Since $T$ extends $\Sigma^1_1\text{-}\mathsf{AC}_0$, by Theorem \ref{folklore}, $|T|_{\mathsf{RE}}=|T|_{\mathsf{AN}}$. Moreover, since $T$ extends $\Sigma^1_1\text{-}\mathsf{AC}_0$, we infer that $\prec_T$ is $T$-provably equivalent to a $\Sigma^1_1$ formula. So $\prec_T$ is a presentation of $|T|_{\mathsf{AN}}$ that is $T$-provably equivalent to a $\Sigma^1_1$ formula, whence $T\nvdash \mathsf{WF}(\prec_T)$.

It follows immediately from (\ref{implies-formula}) and from the claim that $T\nvdash \mathsf{RFN}_{\Pi^1_1}(T)$.
\end{proof}

\subsection{Well-foundedness}\label{foundation}

In this subsection we prove a strengthening of Theorem \ref{main-new} that is of independent interest. The following result is proved in \cite{pakhomov2018reflection, pakhomov2021reflection}:

\begin{theorem}[Pakhomov--W.]\label{old-pw}
There is no sequence $(T_n)_{n<\omega}$ of $\Pi^1_1$-sound and $\Sigma^0_1$-definable extensions of $\mathsf{ACA}_0$ such that for each $n$, $T_n\vdash \mathsf{RFN}_{\Pi^1_1}(T_{n+1})$.
\end{theorem}

Pakhomov and the author proved Theorem \ref{old-pw} to provide an explanation of the apparent pre-well-ordering of natural theories by proof-theoretic strength; see \cite{walsh2021hierarchy} for a discussion of this phenomenon. In \cite{pakhomov2018reflection, pakhomov2021reflection}, Theorem \ref{old-pw} is proved using G\"{o}del's second incompleteness theorem. In particular, we show that the theory $\mathsf{ACA}_0+\varphi$, where $\varphi$ states that Theorem \ref{old-pw} is false, proves its own consistency. In \cite{lutz2020incompleteness} it is claimed that such a result ``could be proved by showing that a descending sequence $(T_n)_{n<\omega}$ of theories would induce a descending sequence in the ordinals (namely, the associated sequence of
proof-theoretic ordinals).'' We now present such a proof (though for $\Sigma^1_1$-definable extensions of $\Sigma^1_1\text{-}\mathsf{AC}_0$ rather than for $\Sigma^0_1$-definable extensions of $\mathsf{ACA}_0$). 

What follows is a restatement of Theorem \ref{stronger}:

\begin{theorem}\label{new-pw}
There is no sequence $(T_n)_{n<\omega}$ of $\Pi^1_1$-sound and $\Sigma^1_1$-definable extensions of $\Sigma^1_1\text{-}\mathsf{AC}_0$ such that for each $n$, $T_n \vdash \mathsf{RFN}_{\Pi^1_1}(T_{n+1})$.
\end{theorem}

\begin{proof}
Suppose that there is such a sequence $(T_n)_{n<\omega}$. From Lemma \ref{key-lemma} we infer that, for each $n$, there is a $\Sigma^1_1$ presentation $\prec_{T_n}$ of $|T_{n}|_{\mathsf{RE}}$ such that: $$\Sigma^1_1\text{-}\mathsf{AC}_0\vdash \mathsf{RFN}_{\Pi^1_1}(T_n)\to \mathsf{WF}(\prec_{T_n}).$$
Since each $T_n$ extends $\Sigma^1_1\text{-}\mathsf{AC}_0$, Theorem \ref{folklore} entails that, for each $n$, there is a $\Sigma^1_1$ presentation $\prec_{T_n}$ of $|T_{n}|_{\mathsf{AN}}$ such that: $$\Sigma^1_1\text{-}\mathsf{AC}_0\vdash \mathsf{RFN}_{\Pi^1_1}(T_n)\to \mathsf{WF}(\prec_{T_n}).$$
Since each $T_n$ extends $\Sigma^1_1\text{-}\mathsf{AC}_0$, for each $n$:
$$T_n\vdash \mathsf{RFN}_{\Pi^1_1}(T_{n+1})\to \mathsf{WF}(\prec_{T_{n+1}}).$$
By assumption, for each $n$, $T_n\vdash\mathsf{RFN}_{\Pi^1_1}(T_{n+1})$, so we infer that, for each $n$:
$$T_n\vdash \mathsf{WF}(\prec_{T_{n+1}}).$$
Whence $|T_n|_{\mathsf{AN}}>|T_{n+1}|_{\mathsf{AN}}$ for each $n$. Yet $|T_n|_{\mathsf{AN}}$ is an ordinal for each $n$. So $(|T_n|_{\mathsf{AN}})_{n<\omega}$ is a descending sequence in the ordinals.
\end{proof}


Note that Theorem \ref{new-pw} entails Theorem \ref{main-new}. Indeed, if $T$ were a counter-example to Theorem \ref{main-new} then we would get a counter-example to Theorem \ref{new-pw} by letting $T=T_n$ for each $n$.



\section{Avoiding Diagonalization}\label{avoiding-diag}

Considering the motivations outlined in \textsection 1, it is desirable to avoid diagoanlization in the proofs of Theorem \ref{bounding} and Theorem \ref{folklore}.

As discussed in \textsection \ref{first-folk}, the standard proofs of Spector's $\Sigma^1_1$-bounding theorem rely on diagonalization. However, there is already an alternate proof of Spector's theorem due to Beckmann and Pohlers \cite{beckmann1998applications} that uses Gentzen's methods and avoids diagonalization.

Theorem \ref{folklore}, on the other hand, relies on Lemma \ref{rathjen}, which is a formalized version of $\Sigma^1_1$-bounding. Rathjen's proof of Lemma \ref{rathjen} uses the recursion theorem to formalize the standard proof of $\Sigma^1_1$-bounding. In this section we develop an alternate proof of Lemma \ref{rathjen}. Rather than attempt to formalize the diagonalization proof of $\Sigma^1_1$-bounding in a different way, we instead formalize the Beckmann--Pohlers proof.

\subsection{Infinitary derivations}

The Beckmann-Pohlers proof of $\Sigma^1_1$-bounding involves the analysis of derivations in a cut-free infinitary proof system. We provide here a standard definition of such a proof system; for other discussion of such proof systems, see \cite{pohlers1989proof, pohlers1998subsystems}. Note that this proof system is a version of the Tait calculus. Thus, our proof system deals with formulas within which negation is only appended to atomic formulas; this is possible due to the normal form theorems available in classical logic. In what follows, let $\mathsf{Diag}(\mathbb{N})$ be the atomic diagram of $\mathbb{N}$ in the signature $(0,1,+,\times)$.

\begin{definition}\label{infinitary}
We define $\vdash^\alpha\Delta$ inductively by the following clauses:
\begin{itemize}
    \item[(AxM)] If $\Delta\cap \mathsf{Diag}(\mathbb{N})\neq\emptyset$, then $\vdash^\alpha \Delta$ for all ordinals $\alpha$.
    \item[(AxL)] If $t^\mathbb{N}=s^\mathbb{N}$, then $\vdash^\alpha \Delta,s\notin X,t\in X$ for all ordinals $\alpha$.
    \item[($\wedge$)] If $\vdash^{\alpha_i}\Delta,A_i$ and $\alpha_i<\alpha$ for $i=1,2$ then $\vdash^\alpha \Delta, A_1\wedge A_2$.
    \item[($\vee$)] If $\vdash^{\alpha_i}\Delta,A_i$ and $\alpha_i<\alpha$ for some $i\in\{1,2\}$ then $\vdash^\alpha \Delta, A_1\vee A_2$.
    \item[($\forall$)] If $\vdash^{\alpha_i}\Delta,A(i)$ and $\alpha_i<\alpha$ for all $i\in\mathbb{N}$, then $\vdash^\alpha \Delta, \forall x A(x)$.
    \item[($\exists$)] If $\vdash^{\alpha_i}\Delta,A(i)$ and $\alpha_i<\alpha$ for some $i\in\mathbb{N}$, then $\vdash^\alpha \Delta, \exists x A(x)$.
\end{itemize}
\end{definition}

The relation $\vdash^\alpha\Delta$ is to be read that there is an infinite proof tree of $\bigvee\Delta$ whose depth is bounded by the ordinal $\alpha$.

Let's briefly record two lemmas that we will make use of. First we state the ``monotonicity lemma,'' which follows immediately from the definition of $\vdash^\alpha \Delta$:
\begin{lemma}\label{monotonicity-first}
If $\vdash^\alpha\Delta$, $\alpha\leq\beta$, and $\Delta\subseteq \Gamma$, then $\vdash^\beta\Gamma$.
\end{lemma}

Second, we state the $\wedge$-inversion rule. For a proof of the $\wedge$-inversion rule, see \cite{pohlers1989proof} Theorem 10.7.
\begin{lemma}\label{wedge-inversion-first}
If $\vdash^\alpha \Delta,\bigwedge \{A_i \mid i\in I\}$ then, for all $i\in I$, $\vdash^\alpha A_i$. 
\end{lemma}

The infinitary proof calculus is sound and complete for $\Pi^1_1$ sentences of arithmetic:

\begin{theorem}
For any $\Pi^1_1$ sentence $\forall\vec{X}\varphi(\vec{X})$:
$$ \mathbb{N}\vDash \forall\vec{X}\varphi(\vec{X})\Leftrightarrow \exists \alpha<\omega_1^{\mathsf{CK}} \; \vdash^\alpha F(\vec{X}).$$
\end{theorem}

In fact, there is a sharp restricted version of the completeness half of the theorem relating consequences of $\mathsf{ACA}_0$ and proofs of height less than $\varepsilon_0$:

\begin{theorem}\label{sigma-fact-first}
For any $\Pi^1_1$ sentence $\forall \vec{X} \varphi(\vec{X})$:
$$\mathsf{ACA}_0\vdash \forall\vec{X}\varphi(\vec{X})\Rightarrow  \exists \alpha<\varepsilon_0 \;  \vdash^\alpha \varphi(\vec{X}).$$
\end{theorem}

Note that the definition of infinitary derivations makes use of transfinite recursion and is beyond the scope $\mathsf{ACA}_0$. Nevertheless, there are many methods for formalizing infinitary derivations in such a way that appropriate versions of Lemma \ref{monotonicity-first}, Lemma \ref{wedge-inversion-first}, and Theorem \ref{sigma-fact-first} are provable in $\mathsf{ACA}_0$, all without recourse to the fixed point lemma or recursion theorem. For present purposes, we will need to formalize only those infinitary derivations whose depth is less than $\varepsilon_0$, a rather meager class of infinitary derivations. We will turn to the specifics in the next subsection.

\subsection{Formalizing infinitary derivations} In this subsection we turn to the task of formalization in $\mathsf{ACA}_0$. This task has two components. First, we must describe how it is that we \emph{define} infintary proofs in $\mathsf{ACA}_0$. Second, we must describe how it is that we \emph{reason} about infintary proofs in $\mathsf{ACA}_0$. A necessary pre-condition for completing both tasks is fixing an ordinal notation system.

\begin{remark}
We \emph{fix} a nice ordinal notation system for ordinals up to and including (at least) $2^{\varepsilon_0}+1$. We use the symbols $\{\boldsymbol{<},\boldsymbol{>},\boldsymbol{\leq},\boldsymbol{\geq}\}$ for this ordinal notation system. In the remainder of this section of the paper, when we use these symbols we are using them to refer to this fixed ordinal notation system.
\end{remark}

The basic idea behind our definition of infinitary proofs in $\mathsf{ACA}_0$ is that infinitary proofs are $\omega$-branching trees. Each node in the proof is \emph{tagged} with a sequent, ordinal notation, and a rule:
\begin{definition}
Let $\mathsf{SEQ}$ be the set of finite sequents, i.e., sets of formulas in (Tait calculus) normal form in the signature $(0,1,+,\times)$. Let $$\mathsf{RULE} =\{\mathsf{AxM}, \mathsf{AxL}, \wedge, \vee, \forall, \exists, \mathsf{CUT}, \mathsf{REP} \}.$$
\end{definition}

We demand that the trees satisfy \emph{local correctness conditions}. The local correctness conditions merely say that if a node is tagged with a sequent $\Delta$ and rule $R$, then the premises of that node are tagged with sequents that are correct for the rule $R$. Buchholz essentially introduces these local correctness conditions (changed only slightly here) in \cite{buchholz1991notation} Definitions 2.1--2.3.

\begin{definition}\label{local}
Let $(\Delta,R)\subseteq \mathsf{SEQ}\times \mathsf{RULE}$ and let $(\Delta)_{i\in I}$ be a sequence of sequents (the premises of $\Delta$). We say that $(\Delta,R)$ and $(\Delta)_{i\in I}$ jointly satisfy the local correctness conditions if each of the following holds:
\begin{itemize}
    \item[(AxM)] If $R=\mathsf{AxM}$ then $\Delta\cap\mathsf{Diag}(\mathbb{N})\neq\emptyset.$
    \item[(AxL)] If $R=\mathsf{AxL}$ then there are $t^\mathbb{N}=s^\mathbb{N}$ such that $s\notin X,t\in X\in \Delta$.
    \item[($\wedge$)] If $R=\wedge$ then $I=\{1,2\}$ and for some $A_1$ and $A_2$: 
    $$ \text{$A_1\wedge A_2\in\Delta$ and for all $i\in\{1,2\}$, $\Delta_i \subseteq \Delta,A_i$.}$$
    \item[($\wedge$)] If $R=\vee$ then $I\subseteq\{1,2\}$ and for some $A_1$ and $A_2$: 
    $$ \text{$A_1\vee A_2\in\Delta$ and for some $i\in\{1,2\}$, $\Delta_i \subseteq \Delta,A_i$.}$$
    \item[($\forall$)] If $R=\forall$ then $I=\mathbb{N}$ and for some $\forall xA(x)$: 
    $$ \text{$\forall xA(x)\in\Delta$ and for all $i\in \mathbb{N}$, $\Delta_i \subseteq \Delta,A(i)$.}$$
    \item[($\exists$)] If $R=\exists$ then $I\subseteq\mathbb{N}$ and for some $\exists xA(x)$: 
    $$ \text{$\exists xA(x)\in\Delta$ and for some $i\in \mathbb{N}$, $\Delta_i \subseteq \Delta,A(i)$.}$$
    \item[(CUT)] If $R=\mathsf{CUT}$ then $I=\{1,2 \}$ and for some $A$: 
    $$ \text{$\Delta_1 \subseteq \Delta,A$ and $\Delta_2\subseteq \Delta,\neg A$.}$$
    \item[(REP)] If $R=\mathsf{REP}$ then $I=\{1\}$ and $ \text{$\Delta_1 = \Delta$.}$
\end{itemize}
\end{definition}

\begin{definition}[$\mathsf{ACA}_0$]\label{inf-proof}
An \emph{infinitary proof} is an $\omega$-branching tree where each node is labeled by a triple $(\Delta,R,\alpha)$ consisting of a sequent $\Delta$, rule $R$, and ordinal notation $\alpha$ such that:
\begin{enumerate}
    \item The ordinal labels strictly descend from the root towards the axioms;
\item The local correctness conditions from Definition \ref{local} are satisfied.
\end{enumerate}
Note that (1) does not force the ordinal tags to be exact but merely to give bounds.

We are particular interested in those infinitary proofs in which the rule $\mathsf{CUT}$ is not applied. We write $\vdash^\alpha_\Delta$ if the sequent $\Delta$ has such an infinitary proof wherein the root has ordinal tag $\alpha$.
\end{definition}

Now we turn to the task of formalizing reasoning about infinitary derivations in $\mathsf{ACA}_0$. One particularly elegant way of formalizing such reasoning is due to Buchholz \cite{buchholz1991notation}. We fix a standard embedding $f$ of proofs of $\Pi^1_1$ statements in $\mathsf{ACA}_0$ into infinitary derivations in $\omega$-logic. Using our ordinal notation system $\boldsymbol<$ that includes a representation of $\varepsilon_0$, we can define a term system for those infinitary derivations that arise from $f$. The term of a proof in this term system encodes the information in its root, i.e., its sequent, the rule it was inferred with, and an ordinal bound. The definition of this term system for infinitary derivations uses primitive recursion but does not use the fixed point lemma or recursion theorem. 

\begin{remark}\label{arith-comp}
An infinitary proof is coded by the label of its root. Buchholz shows that there are primitive recursive functions that can be used to compute, from the code of (the root of) a proof $P$, the codes of $P$'s subtrees. Accordingly, we can use $\mathsf{ACA}_0$ (and even $\mathsf{RCA}_0$) to construct a proof tree from its code. Moreover, we will be able to prove in $\mathsf{ACA}_0$ that the defined tree satisfies the definition of an infinitary proof by the way Buchholz sets up his term system for infinitary derivations.
\end{remark}

In Definition \ref{inf-proof} we did not require that the ordinal tags are exact but merely that they are bounds. Hence, the new version of Lemma \ref{monotonicity-first} is trivial:
\begin{lemma}\label{monotonicity}
For any $\alpha\boldsymbol{<}\varepsilon_0$, $\mathsf{ACA}_0$ proves ``if $\vdash^\alpha\Delta$, $\alpha\boldsymbol{\leq}\beta$, and $\Delta\subseteq \Gamma$, then  $\vdash^\beta\Gamma$.''
\end{lemma}

For the analogue of Lemma \ref{wedge-inversion-first}, we refer the reader to the proof of Theorem 10.7 in \cite{pohlers1989proof}. Note that Pohlers proves Theorem 10.7 by induction along $\alpha$. In the following we can follow suit since we are assuming $\alpha \boldsymbol<\varepsilon_0$.

\begin{lemma}\label{wedge-inversion}
 For any $\alpha\boldsymbol{<}\varepsilon_0$, $\mathsf{ACA}_0$ proves ``if $\vdash^\alpha \Delta,\bigwedge \{A_i \mid i\in I\}$ then, for all $i\in I$,  $\vdash^\alpha A_i$.''
\end{lemma}

Finally, we note that the following version of Theorem \ref{sigma-fact-first} follows easily given how we have set things up:
\begin{theorem}\label{sigma-fact}
For any $\Pi^1_1$ sentence $\forall \vec{X} \varphi(\vec{X})$, if $\mathsf{ACA}_0\vdash \forall\vec{X}\varphi(\vec{X})$ then for some $\alpha\boldsymbol{<}\varepsilon_0$, $ \mathsf{ACA}_0$ proves that $\vdash^\alpha \varphi(\vec{X}).$
\end{theorem}
Indeed, if $\mathsf{ACA}_0\vdash \forall\vec{X}\varphi(\vec{X})$ then, through the usual embedding $f$ of $\mathsf{ACA}_0$ proofs into $\omega$-logic, $\mathsf{ACA}_0$ can prove that there is an infinitary derivation with height $\boldsymbol<\varepsilon_0$ of $\varphi(\vec{X})$. We get a term for this proof in Buchholz's term system and then, by Remark \ref{arith-comp}, we use it to construct an $\omega$-proof of height $\alpha$ of $\varphi(\vec{X})$, all in $\mathsf{ACA}_0$.

Before continuing, we want to note that Lemma \ref{monotonicity}, Lemma \ref{wedge-inversion}, or Theorem \ref{sigma-fact} are all proved without recourse to the recursion theorem or self-reference.

\subsection{The boundedness lemma}

We need to check that a version of what is called ``the boundedness lemma'' is provable in $\mathsf{ACA}_0$. Beckmann and Pohlers claim that a version of the boundedness lemma is already implicit in Gentzen's \cite{gentzen1969provability} proof of the $\mathsf{PA}$ non-derivability of $\varepsilon_0$-induction. To state this result, let us recall one definition that occurs frequently in the work of Pohlers.\footnote{Note that sometimes (e.g., in \cite{pohlers1998subsystems}) a different definition is given and this definition is stated as a theorem. Elsewhere, as in \cite{beckmann1998applications}, the definition given here is used.}
\begin{definition}
The truth complexity $\mathsf{tc}\big(\forall \vec{X}\varphi(\vec{X})\big)$ of a $\Pi^1_1$ statement $\forall \vec{X}\varphi (\vec{X})$ is the least $\alpha$ such that $\vdash^\alpha \varphi(\vec{X})$.
\end{definition}

Before continuing we will also fix some notation. We let:
$$\mathsf{field}(\prec):= \{ x \mid \exists y (x\prec y \vee x \prec y) \} $$
$$\mathsf{Prog}(\prec,X):=\forall x\Big( \big(x \in\mathsf{field}(\prec) \wedge \forall y (y\prec x\to y\in X) \big) \to x\in X \Big)$$
$$\mathsf{TI}(\prec):=\forall X\big(\mathsf{Prog}(\prec,X) \to \forall x \in \mathsf{field}(\prec)\; x\in X\big).$$ 
Note that $\mathsf{TI}(\prec)$ expresses transfinite induction along $\prec$ and for arithmetically definable $\prec$ the sentence $\mathsf{TI}(\prec)$ is $\Pi^1_1$. An upshot of the boundedness lemma is the boundedness theorem, which establishes a tight connection between $\mathsf{otyp}(\prec)$ and $\mathsf{tc}\big(\mathsf{TI}(\prec)\big)$. Beckmann and Pohlers attribute the following consequence of the boundedness theorem to Gentzen:
\begin{theorem}[Gentzen]\label{gentzen-bounding}
For any arithmetic well-ordering $\prec$, $$\mathsf{otyp}(\prec) \leq 2^{\mathsf{tc}\big(\mathsf{TI}(\prec)\big)}.$$
\end{theorem}
Beckmann \cite{beckmann1998applications} has sharpened Gentzen's result to show that $\mathsf{otyp}(\prec)\leq \mathsf{tc}\big(\mathsf{TI}(\prec)\big)$, which he derives from a sharp version of the boundedness lemma. For present purposes, we will not need the sharp version. To state the version that we will need, we need to cover some definitions.

\begin{definition}
A formula $\varphi$ is \emph{$X$-positive} if $\varphi$ has no occurrences of $X$ of the form $t\notin X$.
\end{definition}

\begin{definition}
If $\varphi$ is a formula, then $\varphi[X\mapsto \psi]$ is the set of formulas we get by replacing each occurrence of $t\in X$ in $\varphi$ with $\psi(t)$. If $\Delta=\{\varphi_1,\dots,\varphi_n\}$ is a set of formulas, then $\Delta[X\mapsto \psi]=\{\varphi_1[X\mapsto \psi],\dots,\varphi_n[X\mapsto \psi]\}$.
\end{definition}

\begin{definition}
For any well-ordering $\prec$:
\begin{enumerate}
    \item $|n|_\prec$ is the rank of $n$ in $\prec$;
    \item $\prec_\alpha = \{n\mid |n|_\prec \boldsymbol<\alpha\}.$
\end{enumerate} 
\end{definition}

\begin{remark}\label{remark-ordering}
Note that $y\in X[X\mapsto \prec_\alpha] = |y|_\prec \boldsymbol<\alpha$.
\end{remark}

The following lemma---the boundedness lemma---is a version of Lemma 13.9 in \cite{pohlers1989proof}. Our proof is essentially the same as that in Pohlers, except that Pohlers relies on some notions that are not formalizable in $\mathsf{ACA}_0$. In particular, we are careful to use partial truth-predicates rather than speak of satisfaction in $\mathbb{N}$. 

\begin{lemma}[$\mathsf{ACA}_0$]
Let $\alpha\boldsymbol{<}\varepsilon_0$ be well-founded. Let $\prec$ be an arithmetic well-ordering. Let $\Delta$ be a finite set of $X$-positive formulas. Suppose that:
$$\vdash^\alpha \neg \mathsf{Prog}(\prec,X),t_1\notin X,\dots,t_n\notin X,\Delta.$$ Then it follows that: $$\mathsf{True}_{\Pi^1_1}\Big( \forall X \big(\bigvee\Delta[X\mapsto \prec_{\gamma}] \big) \Big)$$ where $\gamma=\beta+2^\alpha$ and $\beta=\mathsf{max}\{|t_1|_\prec,\dots,|t_n|_{\prec}\}$.
\end{lemma}

\begin{proof}
We prove the claim by induction on $\alpha$; note that this is licit since we are assuming that $\alpha$ is well-founded. We split into cases based on the final inference in the derivation that yields $\vdash^\alpha \neg \mathsf{Prog}(\prec,X),t_1\notin X,\dots,t_n\notin X,\Delta$. Note that we do not have to consider the inference $\mathsf{CUT}$ since the derivation is cut-free. Note that we also do not have to consider the inference $\mathsf{REP}$; if the given proof ends with repetition we simply look at some smaller proof of the same sequent that does not end with repetition.

\textbf{Case 1: The sequent $\neg \mathsf{Prog}(\prec,X),t_1\notin X,\dots,t_n\notin X,\Delta$ is an axiom according to (AxM).} The set $\Delta$ contains a true atomic formula $\varphi$.  Then $\varphi=\varphi[X\mapsto\prec_{\gamma}]$. So $\Delta[X\mapsto \prec_{\gamma}]$ contains a true formula, namely $\varphi=\varphi[X\mapsto\prec_{\gamma}]$.

\textbf{Case 2: The sequent $\neg \mathsf{Prog}(\prec,X),t_1\notin X,\dots,t_n\notin X,\Delta$ is an axiom according to (AxL).} $\Delta$ contains a formula $t_i\in X$ for some $i\leq n$. If $\beta_i =|t_i|_{\prec}$, then $\beta_i\leq\beta<\gamma$ and $\mathsf{True}_{\Pi^1_1}\big( (t_i\in X)[X\mapsto \prec_\gamma] \big)$ since $\beta_i<\gamma$. Hence $\mathsf{True}_{\Pi^1_1}\big( \bigvee\Delta[X\mapsto \prec_\gamma] \big)$.

\textbf{Case 3: The final inference yields $\Delta$.} Assume that the main formula of the final inference belongs to $\Delta$. Then we have the premises:
$$\vdash^{\alpha_i} \neg \mathsf{Prog}(\prec,X),t_1\notin X,\dots,t_n\notin X,\Delta_i$$ where $\Delta_i$ contains only $X$-positive formulas. From the induction hypothesis we infer that $\forall i \; \mathsf{True}_{\Pi^1_1}\Big( \forall X \big(\bigvee\Delta_i[X\mapsto \prec_{\gamma_i}] \big) \Big)$ where $\gamma_i=\beta+2^{\alpha_i}$. Lemma \ref{monotonicity}, i.e., the mototonicity lemma, delivers:
$$\forall i \; \mathsf{True}_{\Pi^1_1}\Big( \forall X \big(\bigvee\Delta_i[X\mapsto \prec_\gamma] \big) \Big).$$ Appealing to Lemma \ref{monotonicity} once again we infer that:
$$\mathsf{True}_{\Pi^1_1}\Big( \forall X \big(\bigvee\Delta[X\mapsto \prec_\gamma] \big) \Big)$$
since validity is preserved by all inferences.

\textbf{Case 4: The final inference yields $\neg \mathsf{Prog}(\prec,X)$.} The main formula of the final inference is:
$$\exists x\Big( \big( x\in\mathsf{field}(\prec) \wedge \forall y(y\prec x \to y \in X) \big) \wedge x\notin X\Big).$$
Then we have the premise:
$$\vdash^{\alpha_0} \neg \mathsf{Prog}(\prec,X),t\in\mathsf{field}(\prec)\wedge\forall y(\neg y\prec t \vee y\in X)\wedge t\notin X,t_1\notin X,\dots,t_n\notin X,\Delta.$$ By $\wedge$-inversion, i.e., Lemma \ref{wedge-inversion}, we obtain:
\begin{equation}\label{inversion-one}
\vdash^{\alpha_0} \neg \mathsf{Prog}(\prec,X), t\in\mathsf{field}(\prec), \forall y(\neg y\prec t \vee y \in X),t_1\notin X,\dots, t_n\notin X,\Delta
\end{equation}
and also:
\begin{equation}\label{inversion-two}
\vdash^{\alpha_0} \neg \mathsf{Prog}(\prec,X),t\notin X,t_1\notin X,\dots,t_n\notin X,\Delta.
\end{equation}

Assume towards a contradiction that $\neg \mathsf{True}_{\Pi^1_1}\Big( \forall X \big(\bigvee\Delta[X\mapsto \prec_{\gamma}] \big) \Big).$

Applying the induction hypothesis to (\ref{inversion-one}) we obtain:
\begin{equation}\label{inversion-three}
\mathsf{True}_{\Pi^1_1}\Big( \forall X \big(\bigvee\Delta[X\mapsto \prec_{2^{\gamma_0}}]\vee \forall y(y\prec t \to y\in X)[X\mapsto \prec_{2^{\gamma_0}}] \big) \Big)
\end{equation}
where $\gamma_0=\beta+2^{\alpha_0}$.

By Lemma \ref{monotonicity}: $$\neg \mathsf{True}_{\Pi^1_1}\Big( \forall X \big(\bigvee\Delta[X\mapsto \prec_{\gamma}] \big) \Big)\text{ entails }\neg \mathsf{True}_{\Pi^1_1}\Big( \forall X \big(\bigvee\Delta[X\mapsto \prec_{2^{\gamma_0}}] \big) \Big).$$ By (\ref{inversion-three}) we then obtain that $y\in \prec_{2^{\gamma_0}}$ for all $y\prec t$, i.e., $|t|_{\prec}\boldsymbol\leq 2^{\gamma_0}$. Letting $\beta_0:=\mathsf{max}\{|t|_\prec,\beta\}$, then we have $\beta_0\boldsymbol\leq\gamma_0$. Applying the induction hypothesis to (\ref{inversion-two}), we obtain:
$$\mathsf{True}_{\Pi^1_1}\Big( \forall X \big(\bigvee\Delta[X\mapsto \prec_{\beta_0 +2^{\alpha_0}}] \big) \Big).$$

Note that $\beta_0\boldsymbol\leq \beta+ 2^{\alpha_0}$ and also $2^{\alpha_0}+2^{\alpha_0} \boldsymbol\leq 2^\alpha$. Hence: $$\beta_0 +2^{\alpha_0} \boldsymbol\leq \beta+2^{\alpha_0} +2^{\alpha_0} \boldsymbol\leq \beta +2^\alpha = \gamma.$$ Lemma \ref{monotonicity} then yields:
$$\mathsf{True}_{\Pi^1_1}\Big( \forall X \big(\bigvee\Delta[X\mapsto \prec_{\gamma}] \big) \Big).$$
Contradicting our initial assumption.
\end{proof}

For the purposes of the present paper, we will appeal only to the following special case of the previous lemma:

\begin{corollary}[$\mathsf{ACA}_0$]\label{formal-bounding}
Let $\alpha\boldsymbol< \varepsilon_0$ be well-founded. Let $\prec$ be an arithmetic well-ordering. Let $\Delta$ be a finite set of $X$-positive formulas. Suppose that:
$$\vdash^\alpha \neg \mathsf{Prog}(\prec,X),\Delta.$$ Then it follows that: $$\mathsf{True}_{\Pi^1_1}\Big( \forall X \big(\bigvee\Delta[X\mapsto \prec_{2^\alpha}] \big) \Big).$$
\end{corollary}


\subsection{Formalizing $\Sigma^1_1$-bounding}

We are now ready to provide a diagonalization-free proof of Lemma \ref{rathjen}, restated here:
\begin{lemma}[Rathjen]
Suppose $H(x)$ is a $\Sigma^1_1$ formula such that:
$$\mathsf{ACA}_0 \vdash \forall x \big(H(x) \to x\in \mathfrak{W}_{Rec} \big).$$
Then for some $e\in Rec$:
$$\mathsf{ACA}_0 \vdash e \in \mathfrak{W}_{Rec} \wedge \neg H(e).$$
\end{lemma}

\begin{proof}
Let $H(x)$ be a $\Sigma^1_1$ formula satisfying the hypothesis of the lemma. Then $H(x)$ is of the form $\exists Y\theta(x,Y)$ for some arithmetic formula $\theta$.For an $x\in \mathfrak{W}_{Rec}$, let $\prec^x$ be the well-ordering encoded by $x$.

We reason as follows:
\begin{flalign*}
\mathsf{ACA}_0 &\vdash \forall x \big(\exists Y\theta(x,Y) \to x\in \mathfrak{W}_{Rec} \big)\\
\mathsf{ACA}_0 &\vdash \forall x \big( \neg \exists Y\theta(x,Y) \vee \forall X\mathsf{TI}(\prec^x,X) \big)\\
\mathsf{ACA}_0 &\vdash \forall x \Big( \neg \exists Y\theta(x,Y) \vee \forall X\big( \neg\mathsf{Prog}(\prec^x,X) \vee \forall y \in\mathsf{field}(\prec^x)\; y\in X \big) \Big)\\
\mathsf{ACA}_0 &\vdash \forall X \forall Y \forall x \big( \neg \theta(x,Y) \vee \neg\mathsf{Prog}(\prec^x,X) \vee \forall y \in\mathsf{field}(\prec^x) \; y\in X \big)
\end{flalign*}

By Theorem \ref{sigma-fact}, there is an $\alpha\boldsymbol<\varepsilon_0$ such that the following is provable in $\mathsf{ACA}_0$:
\begin{equation}\label{star}
    \vdash^\alpha \neg\theta(x,Y), \neg \mathsf{Prog}(\prec^x,X), \forall y \in\mathsf{field}(\prec^x) \; y\in X .
\end{equation}

We now switch to \textbf{reasoning in} $\mathsf{ACA}_0$. Suppose that $H(n)$ holds. That is: 
\begin{equation}\label{existence}
\exists Y \theta(n,Y).
\end{equation}
Applying Corollary \ref{formal-bounding} to (\ref{star}) we infer that:
$$    \forall Y \big( \neg\theta(n,Y) \vee \forall y \in\mathsf{field}(\prec^x)\; y \in \prec^n_{2^\alpha} \big).$$
Which, by definition of $\prec^x_{2^\alpha}$, is just to say:
$$ \forall Y \big( \neg\theta(n,Y) \vee \forall y \in\mathsf{field}(\prec^n)\; y \in \{ k \mid |k|_{\prec^n} \boldsymbol< 2^\alpha \}\big).$$
Which is just to say that:
\begin{equation}\label{universal}
 \forall Y \big( \neg\theta(n,Y) \vee \forall y \in\mathsf{field}(\prec^n)\; |y|_{\prec^n} \boldsymbol< 2^\alpha \big).
 \end{equation}

Combining (\ref{existence}) and (\ref{universal}) we see that $\forall y\in\mathsf{field}(\prec^n) \; |y|_{\prec^n} \boldsymbol< 2^\alpha$.

This is just to say that $\mathsf{otyp}(\prec^n) \boldsymbol< 2^\alpha$. We infer that:
$$\mathsf{sup}\{\mathsf{otyp}(\prec^x) \mid  \mathsf{True}_{\Sigma^1_1}\big(H(x)\big) \}\boldsymbol< 2^\alpha \boldsymbol< 2^\alpha+1 \boldsymbol<\varepsilon_0.$$
Whence $2^\alpha +1\in\mathfrak{W}_{Rec}$ but $\neg H(2^\alpha +1)$.
\end{proof}

\section{Open problems}\label{open-problems}

We will conclude with open problem concerning the sharpness of these theorems. That is, can any of the hypotheses in the statement of these theorems be weakened? There are three hypotheses that can be tweaked in interesting ways. First, there is the question of relaxing the definability condition.

\begin{question}\label{definability-question}
Is there a $\Pi^1_1$-sound and \textbf{$\Pi^1_1$-definable} extension of $\Sigma^1_1\text{-}\mathsf{AC}_0$ that proves its own $\Pi^1_1$-soundness?
\end{question}

Note that a positive answer to this question would imply that Theorem \ref{main-new} is sharp, at least along the dimension of the descriptive complexity of $T$.

If we do not demand that the theory extends $\Sigma^1_1\text{-}\mathsf{AC}_0$, then we can get a positive answer of sorts by considering the set of $\Pi^1_1$ truths. This theory is $\Pi^1_1$-sound and definable by the $\Pi^1_1$ predicate $\mathsf{True}_{\Pi^1_1}(x)$, which says that $x$ encodes a sentence that is $\mathsf{ACA}_0$-provably equivalent to a $\Pi^1_1$ sentence. This theory proves its own $\Pi^1_1$-reflection \emph{statement}:
$$\forall \varphi\in\Pi^1_1 \big( \mathsf{True}_{\Pi^1_1}(\varphi) \to \mathsf{True}_{\Pi^1_1}(\varphi)\big),$$ which is a logical truth. However, note that this depends on treating $\Pi^1_1$-reflection as a single statement, which is arguably inappropriate for a theory that does not extend $\mathsf{ACA}_0$. If we treated $\Pi^1_1$-reflection as a schema:
$$\Bigg\{ \forall \vec{x} \Big( \mathsf{True}_{\Pi^1_1}\big(\varphi(\vec{x})\big) \to \varphi(\vec{x}) \Big) \mid \varphi(\vec{x}) \in \Pi^1_1 \Bigg\}$$
then the theory in question might not prove instances of this schema since $\mathsf{ACA}_0$ is required to transform arbitrary $\Pi^1_1$ statements into normal form.

If we define $T$ as the union of $\Sigma^1_1\text{-}\mathsf{AC}_0$ with the set of all $\Pi^1_1$ truths, then the $\Pi^1_1$-reflection statement:
$$\forall \varphi\in\Pi^1_1 \big( \mathsf{Pr}_{T}(\varphi) \to \mathsf{True}_{\Pi^1_1}(\varphi)\big),$$ is not a $\Pi^1_1$ statement, since the antecedent is $\Pi^1_1$, so it does not trivially follow from $T$. 

Second, there is the question of relaxing the soundness condition.

\begin{question}
Is there a \textbf{$\Sigma^1_1$-sound} and $\Sigma^1_1$-definable extension of $\Sigma^1_1\text{-}\mathsf{AC}_0$ that proves its own $\Pi^1_1$-soundness?
\end{question}

If we demand in addition that the theory proves Theorem \ref{main-new} and \emph{provably} extends $\Sigma^1_1\text{-}\mathsf{AC}_0$, then we get a strong negative answer. Suppose that:
\begin{enumerate}
    \item $T$ is definable by a $\Sigma^1_1$ formula $\tau$;
    \item $T$ proves that $\tau$ extends $\Sigma^1_1\text{-}\mathsf{AC}_0$;
    \item $T$ proves the $\Pi^1_1$-soundness of $\tau$.
\end{enumerate}
Then, since $T$ proves Theorem \ref{main-new}, $T$ proves that $\tau$ is not $\Pi^1_1$-sound. Hence, $T$ proves both that $\tau$ is and is not $\Pi^1_1$-sound, i.e., $T$ is inconsistent, whence $T$ is not $\Sigma^1_1$-sound.

Finally, there is the question of relaxing the condition that $T$ extend $\Sigma^1_1\text{-}\mathsf{AC}_0$.

\begin{question}\label{weaken-theory}
Is there a $\Pi^1_1$-sound and $\Sigma^1_1$-definable \textbf{extension of $\mathsf{ACA}_0$} that proves its own $\Pi^1_1$-soundness?
\end{question}

Regarding this question there are reasons to expect a negative answer. In a recent preprint \cite{aguilera2021pi}, Aguilera and Pakhomov have introduced $|T|_{\Pi^1_2}$, the $\Pi^1_2$ norm of $T$. $|T|_{\Pi^1_2}$ is a dilator associated with $T$ that is in some ways analogous to the proof-theoretic ordinal of $T$, except that it measures the $\Pi^1_2$ consequences of $T$. Aguilera and Pakhomov prove that $\Pi^1_2$-reflection for $T$ is equivalent to the statement ``$|T|_{\Pi^1_2}$ is a dilator'' (\cite{aguilera2021pi} Theorem 7); this is an analogue of Lemma \ref{key-lemma}. The important point is that this result is proved for $T$ extending $\mathsf{ACA}_0$. An appropriate reformulation of their proof of Theorem 7 may deliver a negative answer to Question \ref{weaken-theory}, which would strengthen Theorem \ref{main-new}.


\bibliographystyle{plain}
\bibliography{bibliography}

\end{document}